\documentclass[a4paper,11pt]{amsart}
\usepackage{hyperref,fancyhdr,mathrsfs,amsmath,amscd,amsthm,amsfonts,latexsym,amssymb,stmaryrd}
\usepackage[all]{xy}

\newcommand{\ol}{\mathcal{O}}
\newcommand{\nb}{\mathcal{N}}

\def \a{\alpha}

\def \gf{\mathfrak{g}}

\def \o{\omega}

\def \phi{\varphi}
\def \Phi{\varPhi}

\def \r{\rho}
\def \s{\sigma}

\def \C{\mathbb{C}\,}

\def\widecheckg{g^{\hspace*{-2.5pt}\vbox to 5pt{\hbox to
0pt{\LARGE$\check{}$}}}\hspace*{2pt}}

\def\widecheckl{\lambda^{\hspace*{-3.5pt}\vbox to 8pt{\hbox to
0pt{\LARGE$\check{}$}}}\hspace*{2pt}}

\begin{document}

\title{On the infinitesimal automorphisms\\ 
of principal bundles}
\author{Radu Pantilie}  
\email{\href{mailto:radu.pantilie@imar.ro}{radu.pantilie@imar.ro}}
\address{R.~Pantilie, Institutul de Matematic\u a ``Simion~Stoilow'' al Academiei Rom\^ane,
C.P. 1-764, 014700, Bucure\c sti, Rom\^ania}

\newtheorem{thm}{Theorem}[section]
\newtheorem{lem}[thm]{Lemma}
\newtheorem{cor}[thm]{Corollary}
\newtheorem{prop}[thm]{Proposition}

\theoremstyle{definition}

\newtheorem{defn}[thm]{Definition}
\newtheorem{rem}[thm]{Remark}
\newtheorem{exm}[thm]{Example}

\numberwithin{equation}{section} 

\maketitle
\thispagestyle{empty}
\vspace{-4mm} 
\begin{center}
\emph{This paper is dedicated to the memory of my father - Nicolae Pantilie (1932 - 2017).}
\end{center}
\vspace{1mm} 
\begin{abstract}
We review some basic facts on vector fields, in the complex-analytic setting, thus, obtaining a rationality result and an 
extension of the Birkhoff-Grothendieck theorem, as follows:\\ 
$\bullet$ Let $Z$ be a compact complex manifold endowed with a very ample line bundle $L$\,. Denote by $\gf_L$ the extended Lie algebra 
of infinitesimal automorphisms of $L$\,. If the representation of $\gf_L$ on the space of holomorphic sections of $L$ is irreducible then 
$Z$ is rational.\\ 
$\bullet$ Let $P$ be a holomorphic principal bundle over the Riemann sphere, with structural group $G$ whose Lie algebra 
is not equal to its nilpotent radical. Then there exists a Lie subgroup $H$ of $G$ with the following properties:\\ 
\indent 
{\rm (i)} $H$ is a quotient of a Borel subgroup of ${\rm SL}(2)$\,.\\ 
\indent 
{\rm (ii)} $P$ admits a reduction to $H$. 
\end{abstract}

\section*{Introduction}

\indent 
The (complex) Lie algebras appear, mainly, through their representations but, also, as acting (infinitesimally, through holomorphic vector fields) 
on compact complex manifolds. The two occurrences overlap in the setting provided by compact complex manifolds $Z$ endowed with 
a very ample line bundle $L$\,. Indeed, some of the holomorphic vector fields on $Z$ lift to (extended) infinitesimal automorphisms on $L$ 
and therefore are induced by linear endomorphisms of the space of holomorphic sections of $L$\,. The problem of which holomorphic vector fields 
are obtained this way is, essentially, the same with the existence problem for certain holomorphic partial connections on $L$ \cite{Lieb-81} 
(cf.\ \cite{At-57}\,; see, also, Section \ref{section:Lie_alg_proj}\,, below). It follows (see Theorem \ref{thm:Borel_with_converse}\,) that  
these vector fields are characterised by the fact that their zero set is nonempty, a (classical) result obviously related to the Borel fixed point theorem.\\ 
\indent 
It, also, follows that if the extended Lie algebra $\gf_L$ of infinitesimal automorphisms of $L$ is nontrivial then $Z$ is covered by rational curves 
\cite{Lieb-75} (see Corollary \ref{cor:rational_curves}\,), result which provides a glimpse of the importance of the rational curves in geometry 
(see \cite{Paltin-2005} and the references therein, for more on this). From the differential geometric point of view, we could deduce that if $\gf_L$ 
is not solvable then a Zariski open subset of $Z$ is the twistor space of a $\r$-quaternionic manifold \cite{Pan-qgfs}\,.\\  
\indent 
The study of $\gf_L$ is, also, motivated by the fact that some of the 
basic open problems in geometry are related to the following simple question: \emph{to what extent $Z$ is determined by $\gf_L$\,}? (see, also, 
\cite{HwaMok-prolong} for other results on $\gf_L$\,). After a review, in Section \ref{section:nvp}\,, of nilpotent vector fields on the complex 
projective space, we show (Theorem \ref{thm:rationality}\,) that if the representation of $\gf_L$ on the space of holomorphic sections of $L$ 
is irreducible then $Z$ is rational.\\ 
\indent 
In Section \ref{section:BirGro}\,, we continue the study of extended Lie algebras of infinitesimal automorphisms of holomorphic principal bundles 
by taking $Z$ to be the Riemann sphere but by allowing more general structural groups. This leads to a natural extension (Corollary \ref{cor:BirGro_extended}\,) 
of the Birkhoff--Grothendieck theorem through a natural approach which, also, provides a simple way of determining the normal bundles  
of the Veronese curves (Example \ref{exm:Veronese_normal_bundle}(2)\,).

\section{Nilpotent vector fields on the projective space} \label{section:nvp} 

\indent 
Unless otherwise stated, all the manifolds and maps are assumed complex analytic. If $Z$ is a manifold, we denote by $\ol_Z$ its sheaf of functions 
(which can be seen as the sheaf of sections of the trivial line bundle over $Z$). Further, if ${\rm Pic}\,Z=\mathbb{Z}$ then we denote 
by $\ol_Z(n)$\,, or just $\ol(n)$ if there is no danger of confusion, the (isomorphism class of) line bundle(s) corresponding to $n\in\mathbb{Z}$\,.\\  
\indent 
In this section, we review some (known) facts which will be used later on.\\ 
\indent 
Firstly, let $E$ be a vector space and let $U\subseteq E$ be a vector subspace. Then $PE\setminus PU$ is a tubular neighbourhood 
for any complex projective subspace of $PE$ disjoint from $PU$ and of maximal dimension. Indeed, let $V\subseteq E$ be a vector subspace 
such that $E=U\oplus V$. Then we have a diffeomorphism $PE\setminus PU\to U\otimes\ol_{PV}(1)$\,, intertwining  the projections 
onto $PV$, which to any $[u+v]\in PE$ with $u\in U$ and $v\in V\setminus\{0\}$ associates $u\otimes\a_v$\,, 
where $\a_v\in\bigl(\C\!v\bigr)^*$ is such that $\a(v)=1$\,. It is straightforward to check that this way we have defined a diffeomorphism, as claimed. 
For example, its inverse associates to any $u\otimes\a\in U\otimes\ol_{PV}(1)$\,, over $[v]\in PV$, the point $[\a(v)u+v]\in PE\setminus PU$.\\ 
\indent 
Note that, through this diffeomorphism, $PV$ corresponds to the zero section of $U\otimes\ol_{PV}(1)$\,. Therefore there is no canonical 
(that is, independent of $PV$) diffeomorphism from $PE\setminus PU$ onto $U\otimes\ol_{P(E/U)}(1)$ inducing the above mentioned one. Furthermore, 
any linear Hermitian structure on $E$ such that $U^{\perp}=V$ induces a linear Hermitian structure on $U\otimes\ol_{PV}(1)$ 
and the corresponding bundle of unit open balls, also, corresponds to a tubular neighbourhood of $PV$ into $PE$\,. Moreover, the boundary  
of its closure (essentially, the unit sphere bundle of $U\otimes\ol_{PV}(1)$\,) is a smooth bundle whose space of sections 
can be identified with the subspace $S$ of the unit sphere of $U\otimes V$ 
formed of the decomposable vectors; that is, of the form $u\otimes v$ with $u\in U$ and $v\in V$. Similarly, $S$ is also the sphere bundle 
of $\ol_{PU}(1)\otimes V$ and the projection $S\to PU$ is induced by $E\to U$.\\ 
\indent 
Consequently, we have the following fact. 

\begin{prop} \label{prop:proj_cell_decomp} 
$PE$ is homeomorphic to the bundle of closed unit balls of $U\otimes\ol_{PV}(1)$ attached to $PU$ through the projection from $S$ onto $PU$. 
\end{prop} 

\indent 
We shall, also, need the following two simple lemmas. 

\begin{lem} \label{lem:for_complementary_flags} 
Let $U,V,W\subseteq E$ be vector subspaces such that $V\subseteq W$.\\ 
\indent 
Then $E=U\oplus V$ if and only if $E=U+W$ and $W=(U\cap W)\oplus V$. 
\end{lem} 
\begin{proof} 
If $E=U\oplus V$ then, obviously, $E=U+W$. Hence, the codimension of $U$ in $E$ is equal to the codimension of $U\cap W$ in $W$\,.  
As the former codimension is equal to the dimension of $V$, the spaces $(U\cap W)\oplus V$ and $W$ have the same dimension. 
Together with $(U\cap W)\oplus V\subseteq W$, this proves that $W=(U\cap W)\oplus V$.\\ 
\indent 
Conversely, suppose that $E=U+W$ and $W=(U\cap W)\oplus V$. Because $V\subseteq W$, the latter equality implies that $U\cap V=\{0\}$\,, 
and the fact that the codimension of $U\cap W$ in $W$ is equal to the dimension of $V$. 
But the codimension of $U\cap W$ in $W$ is, also, equal to the codimension of $U$ in $E$. Consequently, $E=U\oplus V$. 
\end{proof} 

\begin{lem} \label{lem:complementary_flags} 
Let $U_1\subseteq\cdots\subseteq U_k$ and $V_1\supseteq\cdots\supseteq V_k$ be \emph{complementary} flags on $E$\,; 
that is, $U_j\oplus V_j=E$\,, for $j=1,\ldots,k$\,, where $k\in\mathbb{N}\setminus\{0\}$\,.\\ 
\indent  
Then $U_j=U_{j-1}\oplus(U_j\cap V_{j-1})$\,, for any $j=2,\ldots,k$\,. 
\end{lem} 
\begin{proof} 
Apply Lemma \ref{lem:for_complementary_flags} with $U=V_{j-1}$\,, $V=U_{j-1}$ and $W=U_j$\,, for $j=2,\ldots,k$\,. 
\end{proof} 

\indent 
Note that, any pair of complementary complete flags of $E$ induces, through Proposition \ref{prop:proj_cell_decomp} 
and Lemma \ref{lem:complementary_flags}\,, the usual cellular decomposition of $PE$\,.\\ 
\indent 
If $A$ is a linear endomorphism on $E$ we shall, also, denote by $A$ the vector field induced on $PE$\,. 
If $A$ is nilpotent then the closure of any of its nontrivial orbits on $PE$ is a rational curve; that is, a nonconstant map from the projective line to $PE$\,.  
More precisely, suppose that $A$ is nilpotent of degree $k+1$\,, where $k\in\mathbb{N}$\,. Then we have two increasing filtrations 
\begin{equation*} 
\begin{split}  
&\{0\}\subseteq{\rm ker}A\subseteq\cdots\subseteq{\rm ker}\bigl(A^k\bigr)\subseteq E\;,\\ 
&\{0\}\subseteq{\rm im}\bigl(A^k\bigr)\subseteq\cdots\subseteq{\rm im}A\subseteq E 
\end{split} 
\end{equation*} 
\noindent 
(the dimensions of each of which determine $A$\,, up to conjugations).  
Thus, if $u\in{\rm ker}\bigl(A^{j+1}\bigr)\setminus{\ker}\bigl(A^j\bigr)$\,, for some $j=1,\ldots,k$\,, 
then $u\,, Au\,,\ldots,A^ju$ are linearly independent, and, hence, the closure of the orbit of $A$ through $[u]\in PE$ 
is a (smooth) normal rational curve of degree $j$\,.   

\begin{prop} \label{prop:orbits_of_nilpotents} 
Let $A$ be a nilpotent endomorphism on the vector space $E$\,, and $k\in\mathbb{N}$\,.\\ 
\indent  
{\rm (i)} The nilpotency degree of $A$ is $k+1$ if and only if the closure of the generic orbit of $A$ on $PE$ is a rational curve of degree $k$\,.\\ 
\indent 
{\rm (ii)} If the nilpotency degree of $A$ is equal to $k+1$\,, the inclusion of $PE\setminus P\bigl({\rm ker}\bigl(A^k\bigr)\bigr)$ into $PE$ 
is homotopic to the map induced by $A^k$. 
\end{prop} 
\begin{proof} 
The first assertion is obvious. To prove (ii)\,, note that, the flow of $A$ on $PE$ defines a rational map from $\C\!P^1\times PE$ to $PE$ 
which is, also, represented by a map $\phi:\C\!P^1\times\bigl(PE\setminus P\bigl({\rm ker}\bigl(A^k\bigr)\bigr)\bigr)\to PE$\,.\\ 
\indent 
For each $z\in\C\!P^1\bigl(=\C\sqcup\{\infty\}\bigr)$\,, let $\phi_z:PE\setminus P\bigl({\rm ker}\bigl(A^k\bigr)\bigr)\to PE$  
be given by $\phi_z(x)=\phi(z,x)$\,, for any $x\in PE\setminus P\bigl({\rm ker}\bigl(A^k\bigr)\bigr)$\,. 
Then $\phi_0$ is just the inclusion of the complement of $P\bigl({\rm ker}\bigl(A^k\bigr)\bigr)$ into $PE$\,, 
whilst $\phi_{\infty}$ is given by the endomorphism $A^k$. 
\end{proof} 

\indent 
We end this section with the following fact. 

\begin{prop} \label{prop:nilpotent_complementary_flags} 
Let $A$ and $B$ be nilpotent endomorphisms of degree $k+1$ on $E$ generating a Lie algebra isomorphic to $\mathfrak{sl}(2)$\,, 
where $k\in\mathbb{N}$\,.\\ 
\indent  
Then ${\rm ker}\bigl(A^k\bigr)$ and ${\rm im}\bigl(B^k\bigr)$ are complementary and the corresponding projection from $E$ onto the latter is given, 
up to a nonzero factor, by $B^k\!A^k$.\\ 
\indent  
Moreover, $U_j={\rm ker}\bigl(A^j\bigr)$\,, $V_j={\rm im}\bigl(B^j\bigr)$\,, $j=1,\ldots,k$\,, define complementary flags  
preserved by $[A,B]$\,. 
\end{prop} 
\begin{proof} 
Denote $H=[A,B]$\,. Then $H$ is semisimple and, by multiplying, if necessary, $A$ (or $B$) with a nonzero constant, we, also, 
have $[H,A]=2A$\,, $[H,B]=-2B$. Consequently, $\bigl[H,A^j\bigr]=2jA^j$\,, $\bigl[H,B^j\bigr]=-2jB^j$ and $\bigl[H,B^j\!A^j\bigr]=0$\,, 
for any $j=1,\ldots,k$\,.\\ 
\indent 
The proof follows. 
\end{proof}

\section{The Lie algebra associated to a projective manifold} \label{section:gf_L} \label{section:Lie_alg_proj} 

\indent  
Recall that it is usual to call a manifold \emph{projective} if it is a compact (complex) submanifold 
of a projective space. Alternatively, this can be described as a manifold $Z$ endowed with a very ample line bundle $L$ (see \cite{GriHar}\,); 
more precisely, an embedding of $Z$ into a projective space is induced by a very ample line bundle over $Z$ if and only if it is normal 
and its image is not contained in a projective hyperplane.\\ 
\indent  
If $L$ is a line bundle over $Z$ (assumed compact) then we have an exact sequence of vector bundles \cite{At-57}  
\begin{equation} \label{e:At-57} 
0\longrightarrow Z\times\C\longrightarrow E\overset{\r}{\longrightarrow}TZ\longrightarrow0\;, 
\end{equation}  
where $E=\frac{T(L^*\setminus0)}{\C\!\setminus\{0\}}$\,. As the space of sections of $E$ can be identitified with the space of 
$\C\!\setminus\{0\}$ invariant vector fields on $L^*\setminus0$ it is endowed with a bracket with respect to which it is a Lie algebra.  
We call this Lie algebra \emph{the Lie algebra associated to $L$} and we denote it by $\gf_L$\,.\\  
\indent 
Note that, $\gf_L$ is just the extended Lie algebra of infinitesimal automorphisms of $L$\,.  
If, further, $L$ is very ample then the dual $V$ of the space of sections of $L$ is a representation space for $\gf_L$\,.\\ 
\indent  
It is useful to understand when a vector field $X$ on $Z$ is induced, through $\r$\,, by an endomorphism of $V$. For this, let $c$ be the (first) Chern class
with complex coefficients of $L$\,. Then $c$ is, also, the obstruction \cite{At-57} for \eqref{e:At-57} to split, and it follows quickly that there exists a section 
$\widetilde{X}$ of $E$ such that $\r\bigl(\widetilde{X}\bigr)=X$ if and only if $\iota_Xc=0$\,; equivalently, $X$ is the restriction to $Z$ of a vector field on $PV$. 
Note that, the endomorphism $A$ on the space of sections of $L$ whose dual induces the vector field on $PV$ 
restricting to $X$ on $Z$ is obtained as follows. 
The fact that $\iota_Xc=0$ is equivalent to the fact that $L$ admits a `partial connection' on $L$\,, over $X$, 
whose covariant derivation we denote by $\nabla_X$\,. 
Then, essentially, as in the proof of \cite[Theorem 1.4]{Lieb-81}\,, we have $As=\nabla_Xs$\,, for any section $s$ of $L$\,. 

\begin{prop} \label{prop:g_Z_contains_nil_endo} 
Let $Z$ be endowed with a very ample line bundle $L$\,. Then the following assertions are equivalent:\\ 
\indent 
{\rm (i)} There exists a nonzero element of $\gf_L$ which is a nilpotent endomorphism of the dual of the space of sections of $L$\,.\\ 
\indent 
{\rm (ii)} $\gf_L$ is not abelian.  
\end{prop} 
\begin{proof} 
Suppose that $\gf_L$ is not abelian and let $\mathfrak{s}$ be the nilpotent radical (see \cite{Bou-Lie_I}\,) of $\gf_L$\,. 
If $\mathfrak{s}\neq\{0\}$ then, obviously, (i) holds. If $\mathfrak{s}=\{0\}$ then $\gf_L$ is reductive and, as it is not abelian, 
(i) holds by the theory of semisimple Lie algebras (see \cite{Hum-80}\,).\\ 
\indent 
If $\gf_L$ is abelian then, as $\gf_L$ is linear algebraic, it is the Lie algebra of an algebraic torus 
(that is, a linear algebraic group isomorphic to the group of diagonal matrices 
in ${\rm GL}(n)$\,, for some $n\in\mathbb{N}\setminus\{0\}$\,; see \cite{Hum-75}\,). Therefore, in this case, (i) cannot hold. 
\end{proof}  

\begin{thm}[see \cite{Lieb-81}\,] \label{thm:Borel_with_converse} 
If $Z\subseteq PV$ is projective then a vector field on $Z$ is induced by an endomorphism of $V$ 
if and only if it has a zero. 
\end{thm}  
\begin{proof} 
The necessity is, essentially, Borel's fixed point theorem (see \cite{Som-73} for a generalization with a simple proof).\\ 
\indent 
For the sufficiency, we may assume the embedding of $Z\subseteq PV$ normal and its image not contained in a projective hyperplane. 
(This will, also, make the corresponding endomorphism unique, up to a multiple of the identity term.)  
Then, if we denote by $L$ the restriction to $Z$ of the dual of the tautological line bundle over $PV$, we have that $V$ is the dual of the space of sections of $L$\,.\\ 
\indent 
We shall, also, need the equivalence of the following three facts, which hold for any holomorphic vector field $X$ on a compact K\"ahler manifold $Z$ 
(see \cite[Theorem 1.5]{Lieb-81} for a longer list of equivalent facts, containing these three):\\ 
\indent 
\quad(i) $X$ has a zero,\\  
\indent 
\quad(ii) $\iota_X\o$ is $\overline{\partial}$-exact, where $\o$ is the K\"ahler form of $Z$,\\ 
\indent 
\quad(iii) $\a(X)=0$\,, for any holomorphic one-form $\a$ on $Z$.\\ 
\indent 
To prove these equivalences, firstly, note that (i)$\Longrightarrow$(iii) is obvious. 
Also, by \cite[Theorem 4.4(2)]{Ko-transf_groups}\,, we have (ii)$\Longleftrightarrow$(iii)\,, 
whilst (ii)$\Longrightarrow$(i) is, essentially, contained in the proof of \cite[Proposition I]{Som-73}\,.\\ 
\indent 
Now, to complete the proof of the theorem, just note that the K\"ahler form of $Z$ represents, up to a nonzero constant factor, the Chern class of $L$ 
with complex coefficients.   
\end{proof} 

\begin{cor}[see \cite{Lieb-81}\,]  
If $Z\subseteq PV$ is projective and with zero first Betti number then any vector field on $Z$ is induced by an endomorphism of $V$. 
\end{cor} 
\begin{proof} 
This follows from Theorem \ref{thm:Borel_with_converse} and its proof. 
\end{proof} 

\begin{cor}[\,\cite{Lieb-75}\,]  \label{cor:rational_curves} 
Let $X$ be a nonzero vector field on a projective manifold $Z$. If the zero set of $X$ is nonempty then through each point of a Zariski open subset of $Z$ 
passes a rational curve. 
\end{cor} 
\begin{proof} 
By Theorem \ref{thm:Borel_with_converse} and Proposition \ref{prop:g_Z_contains_nil_endo}\,, 
either there exists a nonzero nilpotent vector field on $Z$ or $\gf_L$ is the Lie algebra of an algebraic torus.  
Consequently, either there exists a nonzero nilpotent vector field on $Z$ or 
there exists on $Z$ a nontrivial action of $\C\!\setminus\{0\}$\,. The proof quickly follows. 
\end{proof} 
 
\indent 
We end this section with the following result. 

\begin{thm} \label{thm:rationality} 
Let $Z$ be endowed with a very ample line bundle $L$ such that the representation of $\gf_L$ 
on the space of sections of $L$ is irreducible.\\ 
\indent 
Then $Z$ is rational. 
\end{thm} 
\begin{proof} 
We may assume $Z$ of positive dimension, and let $V$ be the dual of the space of sections of $L$\,. 
As the representation of $\gf_L$ on $V$ is irreducible, we have $\dim\gf_L\geq2$\,. 
Hence, from \cite[19.1]{Hum-80}  
we deduce that $\gf_L$ is reductive with center equal to $\C{\rm Id}_V$ and nontrivial semisimple part $\gf\,\bigl(=[\gf_L,\gf_L]\,\bigr)$\,; 
in particular, the representation of $\gf$ on $V$ is irreducible.\\ 
\indent 
Furthermore, an orbit of $\gf$ on $Z\,(\subseteq PV)$ of minimal dimension must be Zariski closed and of positive dimension 
(as the representation of $\gf$ on $V$ is irreducible). Consequently, that orbit is 
determined by a highest weight (with respect to the partial order corresponding to a base of the root system 
determined by a Cartan subalgebra of $\gf$\,) vector. Let $Y\subseteq Z$ be this closed orbit (which is a generalized flag manifold) 
and note that $\gf$ is (canonically isomorphic to) its Lie algebra of vector fields.\\ 
\indent 
Now, choose $\mathfrak{sl}(2)\subseteq\gf$ such that some (and, hence, any) semisimple element of it is a regular element of $\gf$\,, 
and let $k+1$ be the nilpotency degree (with respect to $V$) of its nilpotent elements. Let $A,B\in\mathfrak{sl}(2)$ be nilpotent, linearly independent,   
and denote $H=[A,B]$\,. Then $H$ is regular and, thus, the vector field induced by it on $Y$ has only isolated zeroes.\\ 
\indent 
Note that, $P\bigl({\rm ker}\bigl(A^k\bigr)\bigr)$ cannot contain $Y$ for otherwise the representation of $\gf$ on $V$ would be reducible.    
Thus, $Y$ passes through $PV\setminus P\bigl({\rm ker}\bigl(A^k\bigr)\bigr)$ where the zero set of the vector field determined by $H$ 
is given by $P\bigl({\rm im}\bigl(B^k\bigr)\bigr)$\,. From Proposition \ref{prop:nilpotent_complementary_flags} 
we deduce that $Y$ is contained by a fibre of the rational map induced by $B^k\!A^k$. Therefore the zero set of the vector field induced by $H$ on 
$PV\setminus P\bigl({\rm ker}\bigl(A^k\bigr)\bigr)$ is formed of just one point (as the representation of $\gf$ on $V$ is irreducible) and that point 
is contained by $Z$ (as that point is in $Y$).\\ 
\indent 
Thus, from Proposition \ref{prop:nilpotent_complementary_flags} and the discussion before Proposition \ref{prop:proj_cell_decomp}\,, 
we deduce that we may identify $Z\setminus P\bigl({\rm ker}\bigl(A^k\bigr)\bigr)$ and $\C^{\!n}$, where $\dim Z=n$\,, 
(see, also, \cite[Theorem 4.1]{CarSom-83}\,) and the proof is complete.  
\end{proof}

\section{An extension of the Birkhoff--Grothendieck theorem} \label{section:BirGro} 

\indent 
Let $(P,M,G)$ be a principal bundle. We have an exact sequence of vector bundles 
\begin{equation} \label{e:Atiyah} 
0\longrightarrow {\rm Ad}P\longrightarrow TP/G\longrightarrow TM\longrightarrow0\;, 
\end{equation}  
and the corresponding spaces of sections are, usually, called as follows:\\ 
\indent 
\quad$\bullet$ $H^0({\rm Ad}P)$ is the Lie algebra of infinitesimal automorphisms of $P$,\\ 
\indent 
\quad$\bullet$ $H^0(TP/G)$ is the extended Lie algebra of infinitesimal automorphisms of $P$,\\ 
\indent 
\quad$\bullet$ $H^0(TM)$ is the Lie algebra of vector fields on $M$.\\ 
We denote by $\mu:H^0(TP/G)\to H^0(TM)$ the induced morphism of Lie algebras. Note that, if $M$ is compact then these are 
the opposites of the Lie algebras of the corresponding automorphism groups (see \cite{Mo-58}\,). 

\begin{defn} 
Let $(P,M,G)$ be a principal bundle and let $\mathfrak{h}$ be a Lie subalgebra of $H^0(TM)$\,.\\ 
\indent 
An \emph{$\mathfrak{h}$-invariance} on $P$ is a section of $\mu$ over $\mathfrak{h}$\,, that is, a morphism of Lie algebras 
$\s:\mathfrak{h}\to H^0(TP/G)$ such that $\mu\circ\s=\iota$\,, where $\iota:\mathfrak{h}\to H^0(TM)$ is the inclusion.\\ 
\indent 
A principal bundle over $M$ endowed with an $\mathfrak{h}$-invariance is called $\mathfrak{h}$-invariant. 
In case $\mathfrak{h}=H^0(TM)$ we use the term \emph{equivariance/equivariant}, instead. 
\end{defn} 

\indent 
Similarly to \cite{At-57}\,, a necessary condition for the existence of an $\mathfrak{h}$-invariance 
is the vanishing of an element of $H^1(\mathfrak{h}^*\otimes{\rm Ad}P)$\,.\\ 
\indent 
Any invariance on a principal bundle admits a local description similar to the description of a principal connection 
through local connections forms.\\ 
\indent  
It is easy to see that the principal bundles endowed with invariances form a category. Also, any flat principal connection on a principal bundle 
$P$ over a manifold $M$ endowed with a Lie subalgebra $\mathfrak{h}$ of $H^0(TM)$ determines a $\mathfrak{h}$-invariance on $P$,  
but not all invariances are obtained this way (just look to the principal bundle corresponding to a nontrivial line bundle over the Riemann sphere). 
We have, however, the following fact. 

\begin{rem} \label{rem:rho} 
Let $\r:M\times\mathfrak{h}\to TM$ be the vector bundle morphism defined by $\r(x,X)=X_x$\,, for any $x\in M$ and $X\in\mathfrak{h}$\,. 
Then any $\mathfrak{h}$-invariance corresponds to a unique flat principal $\r$-connection \cite{Pan-qgfs}\,.  
\end{rem}  

\begin{exm} \label{exm:basic_equi} 
1) Let $Z$ be compact and endowed with a very ample line bundle $L$\,. Denote by $\mathfrak{h}=\r\bigl(H^0(E)\bigr)$\,, where $\r:E\to TZ$ is 
as in \eqref{e:At-57}\,. Then $H^0(E)\subseteq\mathfrak{gl}\bigl(H^0(L)\bigr)$ and, for any $X\in\mathfrak{h}$ there exists a unique 
$A_X\in\mathfrak{sl}\bigl(H^0(L)\bigr)$ such that $\r(A_X)=X$. Consequently, on associating $X\mapsto A_X$ we obtain 
a $\mathfrak{h}$-invariance on $L^*\setminus0$\,.\\ 
\indent 
2) Let $H$ be a Lie group and let $\mathfrak{h}$ be its Lie algebra. Let $K\subseteq H$ be a closed subgroup. Then $(TH)/K=(H/K)\times\mathfrak{h}$ 
and $\mu$ restricted to $\mathfrak{h}$ (identified with the space of constant sections of $(H/K)\times\mathfrak{h}$\,) is the canonical 
morphism of Lie algebras from the opposite of $\mathfrak{h}$ to the Lie algebra of vector fields on $H/K$. 
Hence, if $K$ does not contain any normal subgroup of $H$ then $(H,H/K,K)$ is canonically endowed with a $\mathfrak{h}$-invariance, 
which we call \emph{the canonical $\mathfrak{h}$-invariance of $(H,H/K,K)$}\,.\\ 
\indent 
Furthermore, if $\phi:K\to G$ is a morphism of Lie groups then $(H\times_{\phi}G,H/K,G)$ is endowed with a $\mathfrak{h}$-invariance, 
induced by the canonical $\mathfrak{h}$-invariance of $(H,H/K,K)$\,.\\ 
\indent 
3) Conversely, suppose that $H$ is simply-connected and semisimple, $K$ is a parabolic subgroup of $H$ and let $(P,H/K,G)$ be an equivariant principal bundle. 
Then the (infinitesimal) action of $\mathfrak{h}$ on $P$ integrates to an action of $H$ by extended automorphisms.\\ 
\indent   
Let $u_0\in P$ be in the fibre over $K\,(\in H/K)$\,. As $H$ acts by extended automorphisms 
on $P$, so does $K$. Consequently, $\phi:K\to G$ characterised by $a(u_0)=u_0\phi(a)$\,, for any $a\in K$, is a morphism of Lie groups. 
Moreover, $P=H\times_{\phi}G$ and, hence, the orbit of $H$ through $u_0$ is the image of a morphism $\Phi$ of principal bundles from 
$(H,H/K,K)$ to $P$. Therefore the equivariance of $P$ is induced, through $\Phi$, by the canonical equivariance of $(H,H/K,K)$\,. 
\end{exm} 

\begin{thm} \label{thm:equi_exist} 
Any principal bundle over the Riemann sphere, with a structural group whose Lie algebra is not equal to its nilpotent radical, admits an equivariance. 
\end{thm}  
\begin{proof} 
Let $(P,\C\!P^1,G)$ be a principal bundle such that the Lie algebra of $G$ is not equal to its nilpotent radical $\mathfrak{s}$\,. 
Denote by $\mu:H^0(TP/G)\to\mathfrak{sl}(2)$ the induced morphism of Lie algebras. As $\mathfrak{sl}(2)$ is simple,  
it is sufficient to prove that $\mu$ is surjective (see \cite{Bou-Lie_I}\,).\\ 
\indent 
As $\C\!P^1$ is simply-connected, $P$ admits a reduction to the identity component of $G$. Therefore we may assume $G$ connected. 
Note that, the connected Lie subgroup of $G$ whose Lie algebra is $\mathfrak{s}$ is closed. Hence, there exists a submersive morphism of Lie groups 
from $G$ onto a connected reductive Lie group. Consequently, we may assume 
$G$ connected and reductive (because we just have to prove that $\mu$ is surjective).\\ 
\indent 
If $G$ is abelian, the fact that $\mu$ is surjective follows from the proof of Theorem \ref{thm:Borel_with_converse}\,.\\  
\indent  
If $G$ is not abelian then ${\rm Ad}G$ is connected, semisimple. Consequently, we may assume 
$G$ connected and semisimple (and nontrivial). Let $V$ be an irreducible representation of $G$ and denote by $H$ the parabolic subgroup of $G$ 
which is the isotropy subgroup of the closed orbit of $G$ on $PV$.\\ 
\indent 
We claim that $P$ admits a reduction to $H$. Indeed, from a straightforward generalization of \cite[Theorem 4]{Gun-67} 
and the complex version of \cite[Proposition 3.5.3]{PreSeg}\,, 
we deduce that, up to a meromorphic automorphism, $P$ is trivial; equivalently, $P$ admits a meromorphic section. Consequently, also, $P/H$ admits 
a meromorphic section $s$\,. But, $P/H$ is a subbundle of the projectivisation of $P\times_HV$ and, hence, $s$ is, in fact, holomorphic; 
consequently, $P$ admits a reduction to $H$.\\ 
\indent 
Finally, as $H$ is parabolic, there exists a submersive Lie group morphism from $H$ onto $\C\!\setminus\{0\}$\,, and the proof quickly follows.     
\end{proof} 

\begin{cor} \label{cor:BirGro_extended} 
Let $(P,\C\!P^1,G)$ be a principal bundle such that the Lie algebra of $G$ is not equal to its nilpotent radical.\\ 
\indent 
Then there exists a Lie subgroup $H$ of $G$ with the following properties:\\ 
\indent 
\quad{\rm (i)} $H$ is a quotient of a Borel subgroup of ${\rm SL}(2)$\,.\\ 
\indent 
\quad{\rm (ii)} $P$ admits a reduction to $H$. 
\end{cor} 
\begin{proof} 
This is a quick consequence of Theorem \ref{thm:equi_exist} and Example \ref{exm:basic_equi}(3)\,. 
\end{proof}  

\indent 
The notion of invariance can be, obviously, defined for vector bundles, through the corresponding frame bundles, thus,  
obtaining a notion to which the operations of direct sum and tensor product can be applied. Note that, the 
`homogeneous vector bundles' (see \cite{Bott-57}\,) are canonically endowed with equivariances. 

\begin{exm} \label{exm:Veronese_normal_bundle} 
1) Let $H$ be the Borel subgroup of ${\rm SL}(2)$ whose Lie algebra $\mathfrak{h}$ is the orthogonal complement, with respect to the Killing form, 
of the (one-dimensional) subspace of $\mathfrak{sl}(2)$ vanishing on the second element of the canonical basis of $\C^{\!2}$.  
Then $X={\rm diag}(1,-1)$ is a semisimple element of $\mathfrak{h}$\,, where ${\rm diag}\,u$ is the $k\times k$ diagonal matrix 
given by $u\in\C^{\!k}$, $(k\in\mathbb{N}\setminus\{0\})$\,.\\ 
\indent 
Let $U_n$ be the irreducible representation space of ${\rm SL}(2)$ of dimension $n+1$\,, $(n\in\mathbb{N})$\,. 
Then, for any $m,n\in\mathbb{N}$\,, the representation of $H$ with respect to which the equivariant vector bundle associated to 
$\bigl({\rm SL}(2),\C\!P^1,H\bigr)$ is $U_n\otimes\ol(m)$ is given by a representation of $\mathfrak{h}$ 
that associates to $X$ the $(n+1)\times(n+1)$ matrix 
$$m\,{\rm I}_{n+1}+{\rm diag}(n,n-2,\ldots,-n+2,-n)={\rm diag}(m+n,m+n-2,\ldots,m-n+2,m-n)\;.$$ 
\indent 
2) Let $n\in\mathbb{N}$\,, $n\geq2$\,. Then $\ol(n)$ induces a ${\rm SL}(2)$-invariant embedding $\C\!P^1\subseteq PU_n$\,. 
Any Veronese curve is obtained from this embedding by composing it with an element of $PGL(U_n)$\,.\\ 
\indent 
It is obvious that the normal bundle $\nb$ of $\C\!P^1$ into $PU_n$ is equivariant. Furthermore, the exact sequence \eqref{e:At-57}  
applied to the tautological line bundles over $\C\!P^1$ and $PU_n$ give the following equivariant exact sequence  
\begin{equation} \label{e:Veronese_normal_bundle} 
0\longrightarrow U_1\otimes\ol(1)\longrightarrow U_n\otimes\ol(n)\longrightarrow\nb\longrightarrow0\;.   
\end{equation} 
\indent 
Consequently, $\mathcal{N}$ is associated to $\bigl({\rm SL}(2),\C\!P^1,H\bigr)$ through a representation that associates to $X\in\mathfrak{h}$ 
the $(n-1)\times(n-1)$ matrix $${\rm diag}(2n,2n-2,\ldots,4)=(n+2){\rm I}_{n-1}+{\rm diag}(n-2,n-4,\ldots,-n+2)\;.$$ 
\indent 
It follows that $\nb=U_{n-2}\otimes\ol(n+2)$\,. 
\end{exm} 

\begin{rem} \label{rem:BirGro_reformulated} 
The Birkhoff--Grothendieck theorem, as formulated and proved in \cite[Th\'eor\`eme 1.2]{Gro-57}\,, admits the following reformulations:\\ 
\indent 
1) Any principal bundle $P$ over the Riemann sphere, with a structural group whose Lie algebra is reductive, 
admits a unique equivariance whose orbits on $P$ have dimension at most $2$\,.\\ 
\indent 
2) Any principal bundle over the Riemann sphere, with a structural group whose Lie algebra is reductive, 
admits a unique flat $\r$-connection, where $\r$ is given by \eqref{e:At-57} applied to the tautological line bundle over the Riemann sphere. 
\end{rem}


\begin{thebibliography}{10} 


\bibitem{At-57} 
M.~F.~Atiyah, Complex analytic connections in fibre bundles, 
\textit{Trans. Amer. Math. Soc.}, {\bf 85} (1957) 181--207. 

\bibitem{Bott-57} 
R.~Bott, Homogeneous vector bundles, 
\textit{Ann. of Math. (2)}, {\bf 66} (1957) 203--248.

\bibitem{Bou-Lie_I} 
N.~Bourbaki, \textit{\'El\'ements de math\'ematique. Groupes et alg\`ebres de Lie. Chapitre I: Alg\`ebres de Lie}, Seconde \'edition, 
Actualit\'es Scientifiques et Industrielles, No. 1285, Hermann, Paris, 1971.  

\bibitem{CarSom-83} 
J.~B.~Carrell, A.~J.~Sommese, ${\rm SL}(2,\C\!)$ actions on compact K\"ahler manifolds, 
\textit{Trans. Amer. Math. Soc.}, {\bf 276} (1983) 165--179. 

\bibitem{GriHar}
P.~A.~Griffiths, J.~Harris, \textit{Principles of algebraic geometry}, Wiley Classics Library,
John Wiley \& Sons, Inc., New York, 1978. 

\bibitem{Gro-57} 
A.~Grothendieck, Sur la classification des fibr\'es holomorphes sur la sphère de Riemann, 
\textit{Amer. J. Math.}, {\bf 79} (1957) 121--138. 

\bibitem{Gun-67} 
R.~C.~Gunning, \textit{Lectures on vector bundles over Riemann surfaces}, 
University of Tokyo Press, Tokyo, Princeton University Press, Princeton, N.J., 1967. 

\bibitem{Hum-75} 
J.~E.~Humphreys, \textit{Linear algebraic groups}, 
Graduate Texts in Mathematics, 21, Springer-Verlag, New York-Heidelberg, 1975. 

\bibitem{Hum-80}
J.~E.~Humphreys, \textit{Introduction to Lie algebras and representation theory}, 
Third printing, revised, Graduate Texts in Mathematics, 9, Springer-Verlag, New York-Berlin, 1980. 

\bibitem{HwaMok-prolong} 
J.-M.~Hwang, N.~Mok, Prolongations of infinitesimal linear automorphisms of projective varieties 
and rigidity of rational homogeneous spaces of Picard number 1 under K\"ahler deformation, 
\textit{Invent. Math.}, {\bf 160} (2005) 591--645.

\bibitem{Paltin-2005} 
P.~Ionescu, Birational geometry of rationally connected manifolds via quasi-lines, 
\textit{Projective varieties with unexpected properties}, Walter de Gruyter GmbH \& Co. KG, Berlin, 2005, 317--335. 

\bibitem{Lieb-75} 
D.~I.~Lieberman, Holomorphic vector fields on projective varieties, \textit{Several complex variables} 
(Proc. Sympos. Pure Math., Vol. XXX, Part 1, Williams Coll., Williamstown, Mass., 1975), Amer. Math. Soc., Providence, R.I., 1977, 273--276. 

\bibitem{Lieb-81} 
D.~I.~Lieberman, Holomorphic vector fields and rationality, \textit{Group actions and vector fields (Vancouver, B.C., 1981)}, 
Lecture Notes in Math., 956, Springer, Berlin, 1982, 99--117. 

\bibitem{Ko-transf_groups} 
S.~Kobayashi, \textit{Transformation groups in differential geometry}, 
Reprint of the 1972 edition, Classics in Mathematics, Springer-Verlag, Berlin, 1995. 

\bibitem{Mo-58} 
A.~Morimoto, Sur le groupe d'automorphismes d'un espace fibr\'e principal analytique complexe, 
\textit{Nagoya Math. J.}, {\bf 13} (1958) 157--168. 

\bibitem{Pan-qgfs} 
R.~Pantilie, On the embeddings of the Riemann sphere with nonnegative normal bundles, Preprint IMAR, Bucharest, 2013, 
(available from \href{http://arxiv.org/abs/1307.1993}{\tt http://arxiv.org/abs/1307.1993}). 

\bibitem{PreSeg} 
A.~Pressley, G.~Segal, \textit{Loop groups}, Oxford Mathematical Monographs, Oxford Science Publications, 
The Clarendon Press, Oxford University Press, New York, 1986. 

\bibitem{Som-73} 
A.~J.~Sommese, Borel's fixed point theorem for K\"ahler manifolds and an application, 
\textit{Proc. Amer. Math. Soc.}, {\bf 41} (1973) 51--54.



\end{thebibliography}
\end{document}